\documentclass[final,onecolumn]{IEEEtran}

\usepackage{mathtools}
\usepackage{amsfonts}
\usepackage{hyperref}
\usepackage{subcaption}
\usepackage{cite}

\def\tp{\mathsf{T}}

\renewcommand{\u}{\boldsymbol{u}}
\newcommand{\x}{\boldsymbol{x}}
\newcommand{\y}{\boldsymbol{y}}
\newcommand{\z}{\boldsymbol{z}}
\newcommand{\w}{\boldsymbol{w}}
\newcommand{\Nor}{\mathcal{N}}
\newcommand{\N}{\boldsymbol{N}}
\newcommand{\Nat}{\mathbb{N}}

\newcommand{\Tr}{\mathrm{Tr}}
\newcommand{\XTrue}{\mathcal{X}}
\newcommand{\UTrue}{\mathcal{U}}
\newcommand{\Alqr}{\mathcal{A}}
\newcommand{\Blqr}{\mathcal{B}}
\newcommand{\Glqr}{\mathcal{G}}
\newcommand{\Dist}{\mathcal{D}}
\newcommand{\Lower}{\mathfrak{L}}
\newcommand{\Upper}{\mathfrak{U}}

\newtheorem{remark}{\bfseries Remark}
\newtheorem{example}{\bfseries Example}

\newtheorem{proposition}{\bfseries Proposition}

\newtheorem{theorem}{\bfseries Theorem}

\newtheorem{corollary}{\bfseries Corollary}
\newtheorem{lemma}{\bfseries Lemma}

\author{Andrew Lamperski, Khem Raj Ghusinga, and Abhyudai Singh}

\title{Analysis and Control of Stochastic Systems using Semidefinite Programming over Moments}

\begin{document}
\maketitle

\begin{abstract}
  This paper develops a unified methodology for probabilistic analysis
  and optimal control design for jump diffusion processes defined by
  polynomials. 
  For such systems, the evolution of the moments of the state can be
  described via a system of linear ordinary
  differential equations. Typically, however, the moments are not
  closed and an infinite system of equations is required to compute
  statistical moments exactly. 
  Existing methods for stochastic analysis,
  known as closure methods, focus on approximating this infinite system
  of equations with a finite dimensional system. This work develops an
  alternative approach in which the higher order terms, which are
  approximated in closure methods, are viewed as inputs to a
  finite-dimensional linear control system. Under this interpretation, upper
  and lower bounds of statistical moments can be computed via convex linear
  optimal control problems with semidefinite constraints. For analysis
  of steady-state distributions, this optimal control problem reduces
  to a static semidefinite program. These same optimization problems
  extend automatically to stochastic optimal control problems. For
  minimization problems, the methodology leads to guaranteed lower
  bounds on the true optimal value. 
  Furthermore, we show how an approximate optimal
  control strategy can be constructed from the solution of the
  semidefinite program.
  The results are illustrated using numerous examples. 
\end{abstract}

\section{Introduction}
Stochastic dynamics and stochastic optimal control problems arise in a variety
of domains such as finance
\cite{oksendal03,oksendalapplied2009,malliaris82,hanson2007applied}, biology
\cite{allen07,hanson2007applied}, physics \cite{allen07}, robotics
\cite{theodorougeneralized2010},
probabilistic inference
\cite{roberts1996exponential,girolami2011riemann}, and stochastic approximation
algorithms \cite{kushnerstochastic2003}. In stochastic analysis
problems, statistics such as the average size of a biological population
\cite{naasell03,singhderivative2007} or the average density of a gas
\cite{Kuehn16} are desired. Aside from special stochastic processes,
such as linear Gaussian systems, these statistics can rarely be
computed analytically. Typically, estimates of the statistics are
either found through Monte Carlo simulation or approximation
schemes. In stochastic optimal control, a controller attempts to
achieve a desired behavior in spite of noise. For example, a robot
must make reaching movements in the face of noisy actuation
\cite{theodorougeneralized2010}, or the harvest of a fishery must be
managed despite random fluctuations in the fish supply
\cite{alvarezoptimal1998,lunguoptimal1997}. Aside from linear
systems with quadratic costs, few stochastic optimal control problems can be solved
analytically. 

\subsection{Contribution}
This paper defines a unified methodology for approximating both
analysis and control problems via convex optimization methods.
The paper focuses on continuous-time jump diffusions defined by
polynomials. Both finite-horizon and steady-state problems are
considered. In the finite-horizon case, approximations are achieved by
solving an auxiliary linear optimal control problem with semidefinite
constraints. In the steady-state case, this problem reduces to a
static semidefinite program (SDP).

In the case of stochastic analysis problems, our method can
be used to provide provable upper and lower bounds on statistical quantities of interest. For
analysis problems, the method can be viewed as an alternative to moment
closure methods  \cite{socha2008linearization,
    Kuehn16,naasell03,SinghHespanhaLogNormal,SinghHespanhaDM,soltani2015conditional}, which give point approximations to moments, as
  opposed to bounds.

For
stochastic optimal control problems, the auxiliary optimal control
problem is a convex relaxation of the original problem. In the case of
stochastic minimization problems, the relaxation gives provable lower
bounds on the true optimal value. Furthermore, we provide a method for
constructing feasible controllers for the original system. The lower
bound from the relaxation can then be used to compute the optimality
gap of the controller. 

This work is an extension of the conference paper
\cite{lamperski2016stochastic}. That paper only considered
the case of finite-horizon optimal control for stochastic differential
equations. The current work considers both finite-horizon and
steady-state problems. Additionally, the methods are extended to include
jump processes. Furthermore, this paper treats stochastic analysis and
stochastic control problems in a unified manner.

\subsection{Related Work}

For uncontrolled problems, the work is closely related to the moment
closure problems discussed briefly above. The moment closure problem
arises when dynamics of one moment depend on higher order moments,
and so moments of interest cannot be represented by a finite
collection of equations. Moment closure methods utilize probabilistic
or physical principles to replace higher order moments with
approximate values. This work, in contrast, uses the higher order
moments as inputs to an auxiliary linear optimal control problem. A
special case of this idea corresponding to stationary moments of stochastic
differential equations was independently studied in
\cite{KuntzStatMoments16}. 

Most methods for stochastic optimal control focus on dynamic
programming \cite{bertsekasdynamic1995,flemingcontrolled2006} or
approximate dynamic programming
\cite{bertsekasdynamic2012}. More closely related to this work,
however, are relaxation methods. Unlike approximate dynamic
programming, these methods provide bounds on the achievable optimal
values. The most closely related methods focus on occupation measures
\cite{lasserrenonlinear2008,bhatt1996occupation,vinter1993convex}. These
works rely on the general insight that a convex relaxation of an
optimization problem can be obtained by examining randomized
strategies \cite{lasserre2001global}. Such relaxation
methods apply to stochastic optimal control problems, but it is not
clear that they can be used for analysis problems. In contrast, our
relaxation methodology applies to 
both cases. A more removed relaxation method is studied in
\cite{rogers2007pathwise,brown2014information}. These works use
duality to relax the constraint that controllers depend causally on
state information. Also related is the work of
\cite{jumarie1995practical,jumarie1996improvement} which uses a
combination of moment dynamics and deterministic optimal control to
find approximate solutions to stochastic optimal control
problems. This work, however, only considers systems with closed
moments. Furthermore, this work assumes a fixed parametric form for the
controller, and does not obtain a relaxation of the original
stochastic control problem. 

\subsection{Paper Organization}
Section~\ref{sec:prob} defines the general problems of interest. In
this section, a collection of running examples is also
introduced. Background results on stochastic processes are given in
Section~\ref{sec:background}. The main results along with numerical
examples are given in
Section~\ref{sec:results}. Finally, the conclusions are given in
Section~\ref{sec:conclusion}.

\section{Problems}
\label{sec:prob}

This paper develops a methodology for bounding the moments of
stochastic dynamical systems and stochastic control problems in
continuous-time. This section introduces the classes of problems for
which we can apply our new
technique. Subsection~\ref{sec:AnalysisProb} describes the
uncontrolled stochastic systems, Subsection~\ref{sec:ControlProb}
describes the optimal stochastic control problems, and
Subsection~\ref{sec:ss} describes the steady state variants of these
problems. Throughout the section, we will introduce running examples
that illustrate how each of the problems specializes to concrete
systems. 

\subsection{Notation}

Random variables are denoted in bold. 

If $\x$ and $\y$ are identically distributed random variables, we
denote this by $\x \stackrel{d}{=} \y$.

For a stochastic process, $\x(t)$, $d\x(t)$ denotes the increment
$\x(t+dt) - \x(t)$. 

If $\x$ is a random variable, its expectation is denoted by $\langle
\x \rangle$. 

We denote that $\x$ is a Gaussian random variable with mean $\mu$ and
covariance $\Sigma$ by $\x\sim \Nor(\mu,\Sigma)$. More generally, we
write $\x\sim \Dist$ to denote that $\x$ is distributed according to
some distribution $\Dist$.

The notation $M\succeq 0$ denotes that $M$ is a symmetric, positive
semidefinite matrix. If $v$ is a vector, then $v\ge 0$ indicates that
all elements of $v$ are non-negative. 

The set of non-negative integers is denoted by $\Nat$.

\subsection{Bounding Moments of Stochastic Dynamic Systems}
\label{sec:AnalysisProb}

In this subsection, we pose the problem of providing bounds to moments
of stochastic processes. The processes take the form of
continuous-time jump diffusions with non-homogeneous jump rates:
\begin{equation}
\label{eq:uncontrolledDynamics}
d\x(t) = f(\x(t)) dt + g(\x(t)) d\w(t) +
\sum_{i=1}^J (\phi_i(\x(t)) - \x(t))d\N_i(\x(t)).
\end{equation}

Here $\x(t)$ is the state and $\w(t)$ is a Brownian motion with $\w(t)
\sim \Nor(0,tI)$. The terms $d\N_i(\x(t))$ are increments of independent
non-homogeneous Poisson process with jump intensity $\lambda_i(\x(t))$,
so that
\begin{equation}
\label{eq:jumpIncrements}
d\N_i(\x(t)) = \begin{cases}
1 & \textrm{ with probability } \lambda_i(\x(t)) dt
\\
0 & \textrm{ with probability } 1 - \lambda_i(\x(t))dt.
\end{cases}
\end{equation}
It will be assumed that the initial condition is either a constant,
$\x(0) = x_0$, or that $\x(0)$ is a random variable with known
moments.

We assume that all of the functions, $f$, $g$, $\phi_i$ and
$\lambda_i$, are polynomials. Let $c$ and $h$ 
be polynomial functions.

\begin{remark}
Recall that $d\x(t) = \x(t+dt) - \x(t)$ and $d\w(t) \sim \Nor(0,dt I
)$. Thus, the equations 
\eqref{eq:uncontrolledDynamics} and 
\eqref{eq:jumpIncrements} can be used to simulate the system using an
Euler-Maruyama method. However, as written, the equations do not
uniquely specify the behavior of the Poisson processes. In particular, 
over the interval $[t,t+dt]$, more than one of the processes may have
a jump. In this scenario, it 
is not clear which of the jumps occurs first. This ambiguity can be
resolved by instead simulating the following identically distributed
jump process:
\begin{subequations}
  \nonumber
  \begin{align}
    \sum_{i=1}^J \left(
      \phi_i(\x(t)) - \x(t) 
    \right) d\N_i(\x(t)) & \stackrel{d}{=} \z(t) d\N(\x(t)) \\
    d\N(\x(t)) &= \begin{cases}
      1 & \textrm{ with probability } \sum_{i=1}^J \lambda_i(\x(t))dt \\
      0 & \textrm{ with probability } 1 - \sum_{i=1}^J
      \lambda_i(\x(t))dt 
      \end{cases}\\
      \z(t) &= \phi_j(\x(t)) - \x(t) \quad \textrm{with probability}
      \quad
      \frac{\lambda_j(\x(t))}{\sum_{i=1}^J \lambda_i(\x(t))}.
  \end{align}
\end{subequations}
\end{remark}

We assume that there is some quantity of interest that can be
described by
\begin{equation}
  \label{eq:analysisObj}
  \left\langle
  \int_0^T c(\x(t))dt + h(\x(T))
  \right\rangle,
\end{equation}
where $c$ and $h$ are polynomials.
Such a quantity could represent, for example, total energy expenditure
over an interval, or variance of the state at the final time. 

Our first problem, solved in Subsection~\ref{sec:resBounding}, is to 
give a systematic method for
computing an increasing sequence of lower bounds, $\Lower_i$ and a
decreasing sequence  
of upper bounds $\Upper_i$ that satisfy:
\begin{equation}
\label{eq:analysisBounds}
\Lower_0 \le \cdots \le \Lower_i \le 
\left\langle 
\int_0^T c(\x(t)) dt + h(\x(T))
\right\rangle
\le \Upper_i \le \cdots \le \Upper_0.
\end{equation}

Below we will give a few running examples of systems defined by
\eqref{eq:uncontrolledDynamics}, \eqref{eq:jumpIncrements} and
\eqref{eq:analysisObj}. 

\begin{example}[Stochastic Logistic Model]
Consider the stochastic logistic model studied in
\cite{singhderivative2007}. This model is described by:
\begin{equation}
  \label{eq:stochLog}
    \x(t+dt) = \begin{cases}
      \x(t)+1 & \textrm{ with probability } \lambda_1(\x(t)) dt \\
      \x(t)-1 & \textrm{ with probability } \lambda_2(\x(t)) dt \\
      \x(t) & \textrm{ otherwise},
      \end{cases}
\end{equation}
where the jump intensities are given by:
\begin{equation}
\label{eq:stochLogIntensities}
  \lambda_1(x) = a_1 x - b_1 x^2, \qquad \lambda_2(x) = a_2 x + b_2 x^2.
\end{equation}
Here we assume that the coefficients satisfy:
\begin{equation}
\Omega := \frac{a_1}{b_1} \in \Nat, \quad 
a_1 >0, \quad
a_2 >0, \quad
b_1 >0, \quad
b_2 \ge 0.
\end{equation}
These assumptions guarantee that if $\x(0) \in \{0,1,2,\ldots,\Omega\}$
then $\x(t) \in \{0,1,2,\ldots,\Omega\}$ for all $t \ge 0$. We will analyze
the moments of this system via semidefinite programming (SDP). We will see
that the bounds on the process can be incorporated as constraints in
the SDP.

This system is a special case of \eqref{eq:uncontrolledDynamics} with
$f=0$, $g=0$, $\phi_1(x) = x+1$, $\phi_2(x)=x-1$, and $\lambda_i$
defined as above. 

\end{example}

\subsection{Stochastic Optimal Control}
\label{sec:ControlProb}

For optimal control, the systems are similar, except now in addition
to the state process, $\x(t)$, there is also an input process,
$\u(t)$. 

The general problem has the form:

\begin{subequations}
\label{eq:optControl}
\begin{align}
  \label{eq:cost}
& \textrm{minimize} && \left\langle
\int_0^T c(\x(t),\u(t)) dt + h(\x(T),\u(T)) 
\right\rangle \\
& \textrm{subject to} && 
\label{eq:controlledDynamics}
d\x(t) = f(\x(t),\u(t)) dt + g(\x(t),\u(t)) d\w(t) +
\sum_{i=1}^J \left(\phi_i(\x(t),\u(t)) - \x(t)\right) d\N_i(\x(t),\u(t)) \\
&&& d\N_i(\x(t),\u(t)) = \begin{cases}
1 & \textrm{ with probability } \lambda_i(\x(t),\u(t)) dt 
\\
0 & \textrm{ with probability } 1 - \lambda_i(\x(t),\u(t))dt 
\end{cases}
 \\
 \label{eq:controlledInequality}
&&& b_i(\x(t),\u(t)) \ge 0 \quad \textrm{for} \quad i\in
\{1,\ldots,n_b\} \\
&&& \x(0) \sim \Dist \\
&&& \u(t) \textrm{ is admissible.}
\end{align}
\end{subequations}
Here $\x(0)\sim \Dist$ means that $\x(0)$ has some specified initial
distribution, $\Dist$. In the problems studied, $\Dist$ will either be
a Dirac $\delta$ function, specifying a constant initial condition, or
else $\Dist$ will be a distribution with known moments.

By admissible, we mean that $\u(t)$ is measurable with respect to the
filtration $\mathcal{F}(t)$ generated by $\{\x(\tau): 0\le \tau \le
t\}$. 
As with the uncontrolled case, we assume that all of the systems are
polynomial. 

Let $V^*$ be the optimal cost. 
The method to be described in Subsection~\ref{sec:resBounding} 
gives a semidefinite
programming
method for computing an increasing sequence of lower bounds, $\Lower_i$ on
the optimal cost:
\begin{equation}
\label{eq:synthesisBounds}
\Lower_0 \le \cdots \le\Lower_i \le V^*.
\end{equation}

In Subsection~\ref{sec:controller}, we show how to use the result of
the semidefinite programs to produce feasible controllers. The cost 
associated with any feasible controller will necessarily be at least
as high as the optimal cost, and so the corrsponding controller gives
an upper bound $\Upper_i$. So, we must have that $\Lower_i \le V^* \le \Upper_i$. So,
if this gap is small, then the controller must be nearly optimal.

\begin{example}[LQR]
  \label{ex:lqr}
  We present the linear quadratic regulator problem because it is
  well-known, and it fits into the general framework of our
  problem. Unlike most of the problems studied, all of the associated
  optimization problems can be solved analytically in this case.
\begin{subequations}
  \label{eq:lqr}
\begin{align}
&\textrm{minimize} && \left\langle 
\int_0^T \left(\x(t)^\top Q \x(t) 
+ \u(t)^\top R \u(t) 
                      \right) dt \right\rangle  + \left\langle\x(T)^\top \Psi \x(T)
\right\rangle \\
& \textrm{subject to} && d\x(t) = (\Alqr \x(t) + \Blqr \u(t)) dt 
+ \Glqr d\w(t)  \\
&&& \x(0) \sim \Nor(0,\Sigma).
\end{align}
\end{subequations}

Here $\x(0)\sim\Nor(0,\Sigma)$ means that $\x(0)$ is a Gaussian random
variable with mean zero and covariance $\Sigma$.

We examine this problem because the SDP from
Subsection~\ref{sec:resBounding} and the controller construction
method from Subsection~\ref{sec:controller} give exact solutions in
this case. In other cases, they only provide lower and upper bounds,
respectively. 

\end{example}

\begin{example}[Fishery Management]
  \label{ex:fishery}
Consider the modified version of the
optimal fisheries management from
\cite{alvarezoptimal1998,lunguoptimal1997}:
\begin{subequations}
\begin{align}
& \textrm{maximize} && \left<
\int_0^T
\u(t)
dt
\right> \\
& \textrm{subject to} &&  d\x(t) = \left(\x(t)-\gamma \x(t)^2 -
                         \u(t)\right)dt 
+\sigma \x(t) d\w(t) \\
&&& \x(0) = x_0 \\
&&& \x(t) \ge 0 \\
&&& \u(t) \ge 0.
\end{align}
\end{subequations}
Here $\x(t)$ models the population in a fishery and $\u(t)$ models the
rate of harvesting. As in the earlier works, a constraint that
$\x(t)\ge 0$ is required to be physically meaningful. Also, without
this constraint, the optimal strategy would be to set $\u(t) =
+\infty$. The constraint that $\u(t) \ge 0$ encodes the idea that
fish are only being taken out, not put into the fishery. 

The primary difference between this formulation and that of
\cite{alvarezoptimal1998} and \cite{lunguoptimal1997}, is that the
cost is not discounted, and operates over a fixed, finite horizon. 

Note that this is a maximization problem, but this is
equivalent to minimizing the objective multiplied by $-1$. 
\end{example}

\begin{example}[Randomly Sampled Feedback]
  Consider a system in which a state $\x_1(t)$ is controlled by a
  sampled feedback controller, where the sampling times are Poisson
  distributed. A simple version of this system is described below. 
  \begin{subequations}
    \label{eq:sampledSys}
    \begin{align}
      d\x_1(t) &= \x_2(t) dt + d\w(t) \\
      x_2(t+dt) &=
                  \begin{cases}
                    \u(t) & \textrm{ with probability } \lambda dt \\
                    \x_2(t) & \textrm{ with probability } 1-\lambda dt.
                  \end{cases}
    \end{align}
  \end{subequations}
  So, here the state $x_2(t)$ holds the most recent sample of the
  input $\u(t)$. A natural cost for this problem is of the form:
  \begin{equation}\label{eq:sampledCost}
    \left \langle \int_0^T \left(Q \x_1(t)^2 + R \x_2(t)^2\right) dt \right \rangle.
  \end{equation}
      
\end{example}

\subsection{Steady State}
\label{sec:ss}

Both the analysis problem and the control problem from Subsections~\ref{sec:AnalysisProb} and
\ref{sec:ControlProb}, respectively,  deal with computing bounds on
the moments of a 
stochastic process over a finite time horizon. The results for these
problems can be modified in a straight-forward manner to bound
steady-state moments of these problems. In this case, for example, the
optimal control problem becomes:
\begin{subequations}
\label{eq:ssControl}
\begin{align}
& \textrm{minimize} && \lim_{T\to\infty}\left\langle
h(\x(T),\u(T)) 
\right\rangle \\
& \textrm{subject to} && 
\eqref{eq:controlledDynamics} - \eqref{eq:controlledInequality}
\end{align}
\end{subequations}

A special case of this problem was studied using similar methods in
\cite{KuntzStatMoments16}. For a detailed discussion of the
differences between those results and the present results, see 
Remark~\ref{rem:specialCase}, below.

\begin{example}[Jump Rate Control]
  \label{ex:rateControl}
  As another variation, we will consider the case of controlling the
  rate of jumps in a control system. A simple model of this takes
  the form:
  \begin{equation}
    \label{eq:controlledJump}
    \x(t+dt) = \begin{cases}
      \x(t) + d\w(t) & \textrm{ with probability } 1 - \u(t) dt \\
      0 & \textrm{ with probability } \u(t) dt.
      \end{cases}
    \end{equation}

    To get a well-posed optimal control problem, we will penalize a
    combination of $\x(t)$ and $\u(t)$, and further constrain $\u(t)$
    to remain in a compact interval. The formal problem can be stated
    as a special case of \eqref{eq:ssControl}:
    \begin{subequations}
      \label{eq:controlledJumpProb}
      \begin{align}
        & \textrm{minimize} && \lim_{T\to\infty} \left\langle Q \x(T)^2
                               + R \u(T) \right\rangle
        \\
        & \textrm{subject to} && d\x(t) = d\w(t)  -\x(t) d\N(t) \\
        &&& d\N(t) = \begin{cases}
          1 & \textrm{ with probability } \u(t) dt \\
          0 & \textrm{ with probability } 1- \u(t)dt
        \end{cases}
        \\
        &&& 0 \le \u(t) \le \Omega.
      \end{align}
    \end{subequations}
    The constraint that $\u(t) \ge 0$ ensures that the jump rate is
    always non-negative. The upper bound constraint, $\u(t) \le
    \Omega$ is more subtle, and is required to ensure finite jump
    rates. To see why it is required, note that the average-cost
    Hamilton-Jacobi-Bellman (HJB) equation for this problem is given by:
    \begin{equation}
      \label{eq:jumpHJB}
      V^* = \min_{0\le u \le \Omega} \left[
        Q x^2 + Ru + \frac{1}{2} \frac{\partial^2 v(x)}{\partial x^2} +
        (v(0) - v(x))u
        \right].
      \end{equation}
      Here $V^*$ is the optimal value of \eqref{eq:controlledJumpProb}
      and $v(x)$ is the differential cost-to-go function. Analysis of
      \eqref{eq:jumpHJB} shows that an optimal strategy must satisfy:
      \begin{equation}
        \label{eq:optJumpStrategy}
        u = \begin{cases}
          0 & \textrm{ if } R + v(0) - v(x) > 0 \\
          \Omega & \textrm{ if } R + v(0) - v(x) < 0.
          \end{cases}
        \end{equation}

        If no upper bound on $u$ were given, then the optimal strategy
        would be to set $u = +\infty$ if $R+v(0)-v(x) <0$. This
        reduces to an event-triggered control strategy,
        \cite{heemels2012introduction,tabuada2007event} in which an 
        instantaneous jump occurs whenever the state crosses the
        boundary $R+v(0)-v(x)=0$. With the upper bound, the strategy
        approximates an event triggered strategy, but in this case,
        rather than an instantaneous jump, the jump rate goes as high
        as possible when the boundary is crossed. This ensures that a
        jump will occur with high probability.

        It should be noted that while some general properties of the
        optimal strategy can be deduced by examining
        \eqref{eq:jumpHJB}, solving this HJB equation appears to be
        non-trivial. First, it requires knowing the true optimal value
        $V^*$ and then solving a non-smooth partial differential
        equation for $v(x)$. 
\end{example}

\section{Background Results}
\label{sec:background}

This section presents some basic results on expressing dynamics of
moments using linear differential
equations. Subsection~\ref{sec:generators} reviews well-known results on
It\^o's formula and generators for jump processes. These results can
be used to see how functions of a stochastic process evolve
over time. 
Subsection~\ref{sec:momentTraj} specializes the results from
Subsection~\ref{sec:generators} to the problems studied in this paper,
as defined in Section~\ref{sec:prob}. In particular, we will see
that moments of the stochastic process are the solutions of an auxiliary
linear control system. Furthermore, objectives and constraints on
the original system can be encoded using linear mappings of the state
and input of the auxiliary linear
control system. Section~\ref{sec:auxExamples} shows how the auxiliary
linear control
system and the corresponding linear mappings from
Subsection~\ref{sec:momentTraj} can be found in examples. 

\subsection{It\^o's Formula and Generators}
\label{sec:generators}

In this paper, we will examine the behavior of moments of a stochastic
process. The dynamics of the moments can be derived using standard
tools from stochastic processes. For more details, see \cite{oksendal03,hanson2007applied}.

Consider the dynamics from \eqref{eq:controlledDynamics}. Note that
\eqref{eq:uncontrolledDynamics} is a special case of
\eqref{eq:controlledDynamics} with all coefficients of $\u(t)$ set to
zero. The It\^o formula for jump processes 
implies that for any smooth scalar-valued function $h(\x(t))$, the increment is
given by
\begin{multline}\label{eq:ito}
dh(\x(t)) = \frac{\partial h(\x(t))}{\partial x} \left(
f(\x(t),\u(t)) dt + g(\x(t),\u(t)) d\w(t) 
\right)
+\frac{1}{2} \Tr \left(
\frac{\partial^2 h(\x(t))}{\partial x^2} g(\x(t),\u(t)) 
g(\x(t),\u(t))^\top
\right)dt \\
+ \sum_{i=1}^J \left(h(\phi_i(\x(t),\u(t))) - h(\x(t))\right) d\N_i(\x(t),\u(t)).
\end{multline}

Taking expectations and dividing both sides by $dt$ gives the
differential equation:
\begin{multline}
\label{eq:generatorDiffEq}
\frac{d}{dt} \langle h(\x(t)) \rangle
=
\left\langle \frac{\partial h(\x(t))}{\partial x} 
f(\x(t),\u(t)) 
+\frac{1}{2} \Tr \left(
\frac{\partial^2 h(\x(t))}{\partial x^2} g(\x(t),\u(t)) 
g(\x(t),\u(t))^\top
\right) \right\rangle 
\\
+
\left\langle
 \sum_{i=1}^J \left(h(\phi_i(\x(t),\u(t))) - h(\x(t))\right) \lambda_i(\x(t),\u(t))
\right\rangle.
\end{multline}
The differential equation can be expressed compactly as $\frac{d}{dt}
\langle h(\x(t))\rangle = \langle Lh(\x(t),\u(t)) \rangle$, where $L$ is
known as the \emph{generator} of the jump process:
\begin{equation}
\label{eq:generator}
Lh(x,u)=\frac{\partial h(x)}{\partial x} 
f(x,u) 
+\frac{1}{2} \Tr \left(
\frac{\partial^2 h(x)}{\partial x^2} g(x,u) 
g(x,u)^\top
\right)
+
 \sum_{i=1}^J \left(h(\phi_i(x,u)) - h(x)\right) \lambda_i(x,u).
\end{equation}
Say that $h(\x(t))$ is a polynomial. 
Then, since all of the functions, $f$, $g$, $\phi_i$, and
$\lambda_i$ are polynomials, term inside the expectation  on the right
hand side of \eqref{eq:generatorDiffEq} must also be a polynomial.

\subsection{The Auxiliary Linear System}
\label{sec:momentTraj}

This subsection specializes the results from the previous subsection
to the problems defined in Section~\ref{sec:prob}. In particular, we
will see how the dynamics, constraints, and costs can all be studied
in terms of an auxiliary linear control system. It should be noted
that the auxiliary control system has the same basic form regardless
of whether the original system was an uncontrolled system, as
introduced in Subsection~\ref{sec:AnalysisProb}, or a stochastic
optimal control problem, as introduced in
Subsection~\ref{sec:ControlProb}.

The state of our auxiliary control system $\XTrue(t)$ will be a
collection of moments:
\begin{equation}
\label{eq:bigState}
\XTrue(t) = 
\begin{bmatrix}
1\\
\left\langle \x(t)^{(m_1)} \right\rangle\\
\left\langle \x(t)^{(m_2)} \right\rangle\\
\vdots \\
\left\langle \x(t)^{(m_N)} \right\rangle
\end{bmatrix}.
\end{equation}
Here, we use the notation
that if $m=(m_1,m_2,\ldots,m_n)$, then  $\x(t)^{(m)}$ denotes the product:
\begin{equation}
\x(t)^{(m)} = (\x_1(t))^{m_1} (\x_2(t))^{m_2} \cdots (\x_n(t))^{m_n}.
\end{equation}
Note that $1 =
\left\langle \x(t)^{(0)} \right\rangle$, in $\XTrue(t)$. For
simplicity, $\XTrue(t)$ will often include all moments of $\x(t)$ up
to some degree. 

The input to the auxiliary control system will be another collection
of moments:
\begin{equation}
\UTrue(t) = \begin{bmatrix}
\left\langle \x(t)^{(q_1)}\u(t)^{(r_1)} \right\rangle \\
\left\langle \x(t)^{(q_2)}\u(t)^{(r_2)} \right\rangle \\
\vdots \\
\left\langle \x(t)^{(q_P)}\u(t)^{(r_P)} \right\rangle
\end{bmatrix}.
\end{equation}
Note that in the case of an
uncontrolled system, the moments from $\UTrue(t)$ will all take the
form $\left\langle \x(t)^{(q_i)}\right \rangle$ for some moment not
appearing in $\XTrue(t)$.
The exact moments which $\UTrue(t)$ contains will be chosen so that
the Lemmas~\ref{lem:dynamics} -- \ref{lem:vec} below hold.

\begin{lemma}
  \label{lem:dynamics}
  {\it
    Consider the dynamics from \eqref{eq:controlledDynamics} and let
    $\XTrue(t)$ be the vector of moments defined in \eqref{eq:bigState}. There
    exist constant matrices, $A$, and $B$, such that 
    \begin{equation}
      \label{eq:linearDynamics}
      \dot \XTrue(t) = A \XTrue(t) + B\UTrue(t).
    \end{equation}
}
\end{lemma}

\begin{IEEEproof}
This is an immediate consequence of \eqref{eq:generator}, provided
that all of the moments on the right hand side that do not appear in
$\XTrue(t)$ are contained in $\UTrue(t)$. 
\end{IEEEproof}

\begin{remark}
  The result from Lemma~\ref{lem:dynamics} is well known, and commonly
  arises in works on \emph{moment closure} for stochastic
  dynamic systems
  \cite{socha2008linearization,
    Kuehn16,naasell03,SinghHespanhaLogNormal,SinghHespanhaDM,soltani2015conditional}.
  We say that a stochastic process has non-closed moments when the
  dynamics of a given moment  depend on a higher order
  moments. In this case, infinitely many differential equations are
  required to describe any one moment exactly. Moment closure methods
  approximate this infinite set of differential equations
  with a finite set of differential equations. 
  In our notation, $\XTrue(t)$ will
  represent some collection of moments of the dynamic system while $\UTrue(t)$ will
  be higher-order moments that are needed to compute $\XTrue(t)$.
  Moment closure methods prescribe rules to approximate the higher
  order moments $\UTrue(t)$, leading to an approximation of the
  moments in $\XTrue(t)$ through \eqref{eq:linearDynamics}. 

  The work in this paper differs from work on moment closure in
  two ways. The first that moment closure results focus on the
  analysis of uncontrolled stochastic processes, whereas our results
  apply to both uncontrolled and controlled stochastic processes. The
  second difference is that where moment problems provide approximate
  values of a given moment, our method can provide upper and lower
  bounds on the moment. We will see in examples that these bounds can
  be quite tight. For discussion on a related method from
  \cite{KuntzStatMoments16}, which also provides upper and lower
  bounds, see Remark~\ref{rem:specialCase}. 
\end{remark}

\begin{lemma}
  \label{lem:Cost}
  {\it
    There exist constant matrices, $C$, $D$, $H$, and $K$ 
    such that
    \begin{equation}
      \label{eq:linearCost}
      \left\langle
        \int_0^T c(\x(t),\u(t)) dt + h(\x(T),\u(T)) 
      \right \rangle = \int_0^T \left(C\XTrue(t) + D\UTrue(t)\right)
      dt 
      + H \XTrue(T) + K \UTrue(T).
    \end{equation}
  }
\end{lemma}

\begin{IEEEproof}
This follows by the assumption that $c$ and $h$ are polynomials,
provided that all moments from $c$ and $h$ that do not appear in
$\XTrue(t)$ are contained in $\UTrue(t)$
\end{IEEEproof}

\begin{lemma}
\label{lem:LMI}
{\it
Let $v_1(\x(t),\u(t)),\ldots,v_m(\x(t),\u(t))$ be any collection of
polynomials. 
There is an affine matrix-valued function
$M$ such that the following holds:
\begin{equation}
  \label{eq:basicLMI}
\left\langle \begin{bmatrix}
  v_1(\x(t),\u(t)) \\
  \vdots \\
  v_m(\x(t),\u(t))
\end{bmatrix}
\begin{bmatrix}
  v_1(\x(t),\u(t)) \\
  \vdots \\
  v_m(\x(t),\u(t))
\end{bmatrix}^\tp
\right\rangle = M(\XTrue(t),\UTrue(t)) \succeq 0.
\end{equation}
Furthermore, if $b_i(\x(t),\u(t)) \ge 0$ for all $t$, and
$s_1(\x(t),\u(t)),\ldots,s_{m_i}(\x(t),\u(t))$ is a collection of
polynomials, then there is a
different affine matrix-valued function $M_{b_i}$ such that
\begin{equation}
  \label{eq:scaledLMI}
\left\langle b_i(\x(t),\u(t))\begin{bmatrix}
  s_1(\x(t),\u(t)) \\
  \vdots \\
  s_{m_i}(\x(t),\u(t))
\end{bmatrix}
\begin{bmatrix}
  s_1(\x(t),\u(t)) \\
  \vdots \\
  s_{m_i}(\x(t),\u(t))
\end{bmatrix}^\tp
\right\rangle = M_{b_i}(\XTrue(t),\UTrue(t)) \succeq 0.
\end{equation} 
}
\end{lemma}

\begin{IEEEproof}
The existence of an affine $M$ 
follows again by the polynomial assumption, 
provided that all moments that do not appear in
$\XTrue(t)$ are contained in $\UTrue(t)$. The fact that all of the
matrices are positive semidefinite follows because the outer product
of a vector with itself 
must be positive semidefinite. Furthermore, the mean of this outer
product is positive semidefinite by convexity of the cone of semidefinite matrices. Similarly, if the positive semidefinite outer product is
multiplied by a non-negative scalar, then the corresponding mean is still positive
semidefinite. 
\end{IEEEproof}

\begin{remark}
  The LMI from \eqref{eq:basicLMI} must hold for any stochastic
  process for which all the corresponding moments are finite. However,
  there could 
  potentially be values $X$ and $U$ such that $M(X,U)\succeq 0$, but
  no random variables $\x$, $\u$ satisfy $\langle v(\x,\u)
  v(\x,\u)^\top \rangle  = M(X,U)$. Here $v$ corresponds to the vector
  of polynomials
  on the left of \eqref{eq:basicLMI}. See~\cite{lasserre2001global}
  and references therein. 
\end{remark}

In \eqref{eq:scaledLMI} we represented polynomial inequality
constraints by an LMI. An alternative method to represent inequality
constraints is given by the following lemma.

\begin{lemma}
  \label{lem:vec}
  {\it
    Let $b_i(\x(t),\u(t))\ge 0$ be a  polynomial
    inequalities and let $k$ be an odd positive number. There exist constant matrices $J_{b_i}$ and
    $L_{b_i}$ such that
    \begin{equation}
      \label{eq:vectorIneq}
      \left\langle
        \begin{bmatrix}
        b_i(\x(t),\u(t)) \\
        b_i(\x(t),\u(t))^3 \\
        \vdots \\
        b_i(\x(t),\u(t))^k
        \end{bmatrix}
      \right\rangle = J_{b_i} \XTrue(t) + L_{b_i} \UTrue(t) \ge 0.
    \end{equation}
  }
\end{lemma}

\begin{IEEEproof}
  If $b_i(\x(t),\u(t)) \ge 0$, then  $b_i(\x(t),\u(t))^m \ge 0$ for
  any positive integer $m$. Convexity of the non-negative real line
  implies that $\langle b_i(\x(t),\u(t))^m \rangle \ge 0$ as well. The
  result is now immediate as long as all moments included on the left of as
  \eqref{eq:vectorIneq} are in either $\XTrue(t)$ or $\UTrue(t)$.
\end{IEEEproof}

Note that only odd powers are used in \eqref{eq:vectorIneq} since even
powers of $b_i(\x(t),\u(t))$ are automatically non-negative.

\begin{remark}
  Using symbolic computation, the matrices from
  Lemmas~\ref{lem:dynamics} -- \ref{lem:vec} can be computed
  automatically. Specifically, given symbolic forms of
  the costs and dynamics from
  \eqref{eq:optControl}, the monomials from $\XTrue(t)$ in
  \eqref{eq:bigState}, and the polynomials used in the outer products
  of \eqref{eq:basicLMI} -- \eqref{eq:scaledLMI}, the matrices $A$,
  $B$, $C$, $D$, $H$, $K$, $M$, and $M_{b_i}$ from the lemmas can be
  computed efficiently. Similarly, given a symbolic vector of the
  polynomials on the left of \eqref{eq:vectorIneq}, the corresponding
  matrices $J_{b_i}$ and $L_{b_i}$ can be computed. 
  Furthermore, the monomials required for
  $\UTrue(t)$ can be identified.

  While the procedure described above automates the rote
  calculations, user insight is still required to identify which
  polynomials to include in $\XTrue(t)$ and
  the outer product constraints \eqref{eq:basicLMI} -- \eqref{eq:scaledLMI}. We
  will see that natural choices for these polynomials can typically be
  found by examining the associated generator,
  \eqref{eq:generator}. Future work will automate the procedure of
  choosing state and constraint polynomials. 
\end{remark}

\subsection{Examples}
\label{sec:auxExamples}

In this subsection we will discuss how the results of the previous subsection can
be applied to specific examples. We will particularly focus on natural
candidates for the state $\XTrue(t)$ and input $\UTrue(t)$ vectors.

\begin{example}[Stochastic Logistic Model -- Continued]
  \label{ex:logisticState}
Recall the stochastic logistic model from \eqref{eq:stochLog}. This
system is an autonomous pure jump process with two different jumps. So, the
generator, \eqref{eq:generator}, reduces to:
\begin{equation}
Lh(x) = (h(\phi_1(x)) - h(x)) \lambda_1(x) + 
(h(\phi_2(x)) - h(x)) \lambda_2(x),
\end{equation} 
where $\phi_1(x) = x+1$ and $\phi_2(x)=x-1$, and $\lambda_i$ were
defined in \eqref{eq:stochLogIntensities}. Thus, the
differential equations corresponding to the moments have the form:
  \begin{equation*}
    \frac{d}{dt}
    \langle \x(t)^k \rangle = 
    \langle
    ((\x(t)+1)^k-\x(t)^k)(a_1\x(t)-b_1\x(t)^2) + ((\x(t)-1)^k-\x(t)^k)
    (a_2 \x(t)+b_2\x(t)^2)
    \rangle.
  \end{equation*}

  Because of a cancellation in the terms $(x+1)^k-x^k$ and
  $(x-1)^k-x^k$, the degree of the right hand side is $k+1$. It
  follows that moment $k$ will depend on moment $k+1$, but no moments
  above $k+1$. From this analysis, we see that the moments of this
  system are not closed. In this case, the auxiliary control variable
  $\UTrue(t)$ will represent a higher order moment that is not part of
  the state, $\XTrue(t)$. 
  
  We further utilize the following constraint, which must be satisfied
  by all valid moments.
  \begin{equation}
    \label{eq:scalarAnalysisOuter}
    \left
    \langle
    \begin{bmatrix}
      1 \\ \x(t) \\ \x(t)^2 \\ \vdots \\
      \x(t)^d
    \end{bmatrix}
    \begin{bmatrix}
      1 \\ \x(t) \\ \x(t)^2 \\ \vdots  \\
      \x(t)^d
    \end{bmatrix}^\top
    \right
    \rangle
    =
    \left
    \langle
    \begin{bmatrix}
      1 & \x(t) & \x(t)^2 & \cdots & \x(t)^d \\
      \x(t) & \x(t)^2 & \x(t)^3 & \cdots & \x(t)^{d+1} \\
      \x(t)^2 & \x(t)^3 & \x(t)^4 & \cdots & \x(t)^{d+2} \\
      \vdots & \vdots & & \ddots & \vdots \\
      \x(t)^d & \x(t)^{d+1} & \x(t)^{d+2} & \cdots & \x(t)^{2d}
    \end{bmatrix}
    \right
    \rangle \succeq 0.
  \end{equation}
   We choose the state and input of the moment dynamics as
  \begin{equation}
    \XTrue(t) = 
\left\langle
  \begin{bmatrix}
      1 \\
      \x(t) \\
      \vdots \\
      \x(t)^{2d-1}
      \end{bmatrix}
      \right\rangle,
      \qquad
      \UTrue(t) = \langle \x(t)^{2d} \rangle.
  \end{equation}
  Since the dynamics of moment $k$ depend moments of order $k+1$ and
  below, it follows that $\UTrue(t)$ is only needed for the equations
  of the final state moment, $\langle \x(t)^{2d-1}\rangle$. 

  As discussed above, if $\x(0)$ is an integer between $0$ and
  $\Omega=\frac{a_1}{b_1}$, then $\x(t)$ is an integer in $[0,\Omega]$ for
  all time. To maintain this bounding constraint, we utilize two
  further constraints:
  \begin{equation*}
        \left
    \langle
    \x(t)
    \begin{bmatrix}
      1 \\ \x(t) \\ \x(t)^2 \\ \vdots \\
      \x(t)^{d-1}
    \end{bmatrix}
    \begin{bmatrix}
      1 \\ \x(t) \\ \x(t)^2 \\ \vdots  \\
      \x(t)^{d-1}
    \end{bmatrix}^\top
    \right
    \rangle \succeq 0,
    \quad
            \left
    \langle
    (\Omega-\x(t))
    \begin{bmatrix}
      1 \\ \x(t) \\ \x(t)^2 \\ \vdots \\
      \x(t)^{d-1}
    \end{bmatrix}
    \begin{bmatrix}
      1 \\ \x(t) \\ \x(t)^2 \\ \vdots  \\
      \x(t)^{d-1}
    \end{bmatrix}^\top
    \right
    \rangle \succeq 0.
  \end{equation*}
   These constraints can also be written in terms of our state
  $\XTrue(t)$ and input $\UTrue(t)$. 
\end{example}

\begin{example}[Fishery Management -- Continued]
  \label{ex:fisheryState}
  Recall the fishery management problem from
  Example~\ref{ex:fishery}. In this problem, the moment dynamics are
  given by:
  \begin{equation}
    \label{eq:fisheryMoments}
    \frac{d}{dt}\langle \x(t)^k \rangle=
    k\left\langle \x(t)^k  - \gamma \x(t)^{k+1} -
      \x(t)^{k-1} \u(t)\right\rangle  +
    \frac{1}{2} k(k-1) \langle \x(t)^k \rangle.
  \end{equation}
   From this equation, we see that moment $k$ depends on moment $k+1$,
  so the moments are not closed. Furthermore, moment $k$ depends on
  the correlation of $\x(t)^{k-1}$ with the input $\u(t)$. The moments
  of $\x(t)$ and $\u(t)$ must satisfy:
  \begin{equation}
    \label{eq:fisheryLMI}
    \left\langle
    \begin{bmatrix}
      1 \\
      \x(t) \\
      \vdots \\
      \x(t)^d \\
      \u(t)
    \end{bmatrix}
    \begin{bmatrix}
      1 \\
      \x(t) \\
      \vdots \\
      \x(t)^d \\
      \u(t)
    \end{bmatrix}^\top 
  \right\rangle
  =
  \left\langle
  \begin{bmatrix}
    1 & \x(t) & \cdots & \x(t)^d & \u(t) \\
    \x(t) & \x(t)^2 & \cdots & \x(t)^{d+1} & \x(t) \u(t) \\
    \vdots & \vdots & \ddots & \vdots & \vdots \\
    \x(t)^d & \x(t)^{d+1} & \cdots & \x(t)^{2d} & \x(t)^d \u(t) \\
    \u(t) & \x(t) \u(t) & \cdots & \x(t)^d \u(t) & \u(t)^2
  \end{bmatrix}
  \right\rangle 
    \succeq 0. 
  \end{equation}

  We represent the state and control for the auxiliary control system
  as:
  \begin{equation}
    \label{eq:fisheryVectors}
    \XTrue(t) = 
    \begin{bmatrix}
      1 \\
      \x(t) \\
      \vdots \\
      \x(t)^{k}
    \end{bmatrix},\quad
    \UTrue(t)  = \begin{bmatrix}
      \x(t)^{k+1} \\
      \vdots \\
      \x(t)^{2d} \\
      \u(t) \\
      \x(t) \u(t) \\
      \vdots \\
      \x(t)^d\u(t) \\
      \u(t)^2
      \end{bmatrix}.
    \end{equation}
    Analysis of \eqref{eq:fisheryMoments} and \eqref{eq:fisheryLMI}
    shows that the largest value of $k$ that can be used in
    \eqref{eq:fisheryVectors} is given by
    \begin{equation}
      k = \begin{cases}
        1 & \textrm{ if } d = 1 \\
        d+1 & \textrm{ if } d > 1.
      \end{cases}
    \end{equation}
\end{example}

\begin{example}[Jump Rate Control -- Continued]
  \label{ex:jumpAux}
  Recall the jump rate control problem from Example~\ref{ex:rateControl}. In this case, the moments have
dynamics given by:
\begin{equation}
  \frac{d}{dt} \langle \x(t)^k \rangle = \frac{1}{2}k(k-1) \langle
  \x(t)^{k-2}\rangle   - \langle \x(t)^k \u(t) \rangle
\end{equation}
for $k\ge 1$. Thus, the state moments do not depend on higher-order
moments of the state, but they do depend on correlations with the
input. 
For this problem, we take our augmented state and
input to be:
\begin{equation}
  \XTrue(t) =
  \left\langle
  \begin{bmatrix}
    1 \\ \x(t) \\ \vdots \\ \x(t)^{2d}
  \end{bmatrix}
\right\rangle,
\qquad
\UTrue(t) =
\left\langle
  \begin{bmatrix}
    \u(t) \\ \x(t)\u(t) \\ \vdots \\ \x(t)^{2d}\u(t)
  \end{bmatrix}  
\right\rangle.
\end{equation}
If $\XTrue(t)$ defines valid moments, it must satisfy
\eqref{eq:scalarAnalysisOuter}. In contrast with the
stochastic logistic model, which constrains the state, in this example
we constrain the input as $0\le \u(t) \le \Omega$. We enforce a
relaxation of this 
constraint using the following LMIs:
  \begin{equation*}
        \left
    \langle
    \u(t)
    \begin{bmatrix}
      1 \\ \x(t) \\ \x(t)^2 \\ \vdots \\
      \x(t)^{d}
    \end{bmatrix}
    \begin{bmatrix}
      1 \\ \x(t) \\ \x(t)^2 \\ \vdots  \\
      \x(t)^{d}
    \end{bmatrix}^\top
    \right
    \rangle \succeq 0,
    \quad
            \left
    \langle
    (\Omega-\u(t))
    \begin{bmatrix}
      1 \\ \x(t) \\ \x(t)^2 \\ \vdots \\
      \x(t)^{d}
    \end{bmatrix}
    \begin{bmatrix}
      1 \\ \x(t) \\ \x(t)^2 \\ \vdots  \\
      \x(t)^{d}
    \end{bmatrix}^\top
    \right
    \rangle \succeq 0.
  \end{equation*}
\end{example}

\section{Results}
\label{sec:results}

This section presents the main results of the paper. In
Subsection~\ref{sec:resBounding} we will show the finite horizon
problems posed in Subsections~\ref{sec:AnalysisProb} and
\ref{sec:ControlProb} can be bounded in a unified fashion by solving
an optimal control problem for the auxiliary linear system. For the
analysis problem from Subsection~\ref{sec:AnalysisProb}, the method
can be used to provide upper and lower bounds on the desired
quantities. For the stochastic control problem of
Subsection~\ref{sec:ControlProb}, the method provides lower bounds on
the achievable values of objective. Subsection~\ref{sec:ssBounds}
gives the analogous bounding results for the steady-state problem
introduced in Subsection~\ref{sec:ss}. Subsection~\ref{sec:controller}
provides a method for constructing feasible control laws for
stochastic control problems using the results of the auxiliary optimal control
problem. Finally, Subsection~\ref{sec:num} applies the results of this
section to the running examples.

\subsection{Bounds via Linear Optimal Control}
\label{sec:resBounding}

This section presents the main result of the paper. The result uses
Lemmas~\ref{lem:dynamics}, \ref{lem:Cost}, and \ref{lem:LMI}
to define an optimal control problem with positive
semidefinite constraints. This optimal control problem can be used to
provide the bounds on analysis and synthesis problems, as described in \eqref{eq:analysisBounds} and
\eqref{eq:synthesisBounds}, respectively. 

\begin{theorem}
  \label{thm:finiteHorizon}
{\it 
  Let $A$, $B$, $C$, $D$, $H$, $K$, $M$, and $M_{b_i}$ be 
  the terms defined in
  Lemmas~\ref{lem:dynamics}~--~\ref{lem:LMI}. Consider
  the corresponding continuous-time semidefinite program:
  \begin{subequations}
    \label{eq:SDP}
    \begin{align}
      \label{eq:SDPCost}
      & \underset{X(t),U(t)}{\textrm{minimize}} && \int_0^T \left(C X(t) + D U(t)\right)
      dt 
      + H X_T + K U_T \\
      & \textrm{subject to} && \dot X(t) = A X(t) + B U(t) \\
      &&& X(0) = \XTrue(0) \\
      &&& M(X(t),U(t)) \succeq 0 \quad \textrm{for all} \quad t\in
          [0,T] \\
      \label{eq:ineqLMI}
      &&& M_{b_i}(X(t),U(t)) \succeq 0 \quad \textrm{for all}
          \\ &&&\quad
      t\in [0,T], \; i\in \{1,\ldots,n_b\}.
    \end{align}
  \end{subequations}

  The optimal value for this problem is always a lower bound
  on the optimal value for the stochastic control problem,
  \eqref{eq:optControl}, provided that the corresponding moments exist. 
  If the number of constraints in the
  linear problem, \eqref{eq:SDP}, is increased, either by adding more moments to
  $\XTrue(t)$ or by adding more semidefinite constraints, the value of 
  \eqref{eq:SDP} cannot decrease.
}
\end{theorem}

\begin{IEEEproof}
  From classical results in stochastic control, 
  optimal strategies can be found of 
  the form $\u(t) = \gamma^*(t,\x(t))$, \cite{flemingcontrolled2006}.
  Thus, the function $\gamma^*$ induces a joint distribution over the
  trajectories $\x(t)$ and $\u(t)$. 
  Denote the corresponding moments by $\XTrue^*(t)$ and
  $\UTrue^*(t)$, provided that they exist. 
  Lemmas~\ref{lem:dynamics} and \ref{lem:LMI} imply that
  $\XTrue^*(t)$ and $\UTrue^*(t)$ satisfy all of
  the constraints of the semidefinite program. Furthermore,
  Lemma~\ref{lem:Cost} implies that the optimal cost of the original 
  stochastic control problem is given by \eqref{eq:linearCost} /
  \eqref{eq:SDPCost} applied to $\XTrue^*(t)$ and $\UTrue^*(t)$. 
  Thus the optimal cost of the semidefinite program is a lower bound
  on the cost of the original stochastic control problem. 

  Say now that the problem is augmented by adding more
  semidefinite constraints. In this case, the set of feasible
  solutions can only get smaller, and so the optimal value cannot decrease. 

  On the other hand, say that more moments are added to $\XTrue(t)$
  giving rise to larger state vectors, $X(t)$. In this case, the
  number of variables increases, but the original variables must still
  satisfy the constraints of the original problem, with added
  constraints imposed by the new moment equations. So, again the
  optimal value cannot decrease. 
\end{IEEEproof}

\begin{remark}
An analogous optimal control problem can be formulated in which the
LMI from \eqref{eq:ineqLMI} is replaced by linear constraints of  the
form:
\begin{equation}
  J_{b_i} X(t) + L_{b_i} U(t) \ge 0,
\end{equation}
where $J_{b_i}$ and $L_{b_i}$ are the matrices from
Lemma~\ref{lem:vec}.

\end{remark}

Now we describe how Theorem~\ref{thm:finiteHorizon} can be used to
provide upper and lower bounds on analysis problems. 

\begin{corollary}
  {\it 
    For the case of the uncontrolled system,
    \eqref{eq:uncontrolledDynamics}, the solution to the
    continuous-time semidefinite program gives a lower bound on the
    following mean:
    \begin{equation}
      \left\langle 
\int_0^T c(\x(t)) dt + h(\x(T))
\right\rangle.
    \end{equation}
    Furthermore, an upper bound on the mean can be found by maximizing the
    objective, \eqref{eq:SDPCost}. 
  }
\end{corollary}

Note that since the objective, \eqref{eq:SDPCost}, is linear, both the
maximization and minimization problems are convex. We will see in
examples that as we increase the size of the SDP, the problem becomes
more constrained and the bounds can be quite tight.

\begin{example}[LQR -- Continued]
  \label{ex:lqrSol}
  We will now show how the linear quadratic regulator problem, defined
  in \eqref{eq:lqr} can be cast into the SDP framework 
  from Theorem~\ref{thm:finiteHorizon}. Denote the joint second moment
  of the state and control by:
  \begin{equation}
    \left\langle
      \begin{bmatrix}
        \x(t)\x(t)^\top & \x(t) \u(t)^\top \\
        \u(t)\x(t)^\top & \u(t) \u(t)^\top
      \end{bmatrix}
    \right\rangle
      =
      \begin{bmatrix}
        P_{xx}(t) & P_{xu}(t) \\
        P_{ux}(t) & P_{uu}(t)
      \end{bmatrix}.
    \end{equation}
    In this case the augmented states and inputs can be written as:
    \begin{equation}
      \XTrue(t) = \begin{bmatrix}
        1 \\ P_{xx}(t)^s
      \end{bmatrix},\qquad
      \UTrue(t) = \begin{bmatrix} P_{xu}(t)^s
        \\ P_{uu}(t)^s
        \end{bmatrix}.
      \end{equation}
      Here $M^s$ denotes the vector formed by stacking the columns of
      $M$. Then the SDP from \eqref{eq:SDP} is equivalent to the
      following specialized SDP:
      \begin{subequations}
        \label{eq:lqrSDP}
      \begin{align}
        & \underset{X(t),U(t)}{\textrm{minimize}}
        && \int_0^T \left(\Tr(QP_{xx}(t)) +\Tr(RP_{uu}(t))\right) dt
           + \Tr(\Psi P_{xx}(T)) \\
        & \textrm{subject to}
        && \dot P_{xx}(t) =
           \Alqr P_{xx}(t) + P_{xx}(t) \Alqr^\top +
           \Blqr P_{ux}(t) + P_{xu}(t) \Blqr^\top + \Glqr
           \Glqr^\top \\
        &&& P_{xx}(0) = \Sigma \\
        &&&       \begin{bmatrix}
        P_{xx}(t) & P_{xu}(t) \\
        P_{ux}(t) & P_{uu}(t)
      \end{bmatrix} \succeq 0.
      \end{align}
    \end{subequations}
       A straightforward calculation from the Pontryagin Minimum
    Principle shows that the optimal solution is given by
    \begin{subequations}
    \begin{align}
      P_{ux}(t) &=K(t) P_{xx}(t)  \\
      P_{xx}(t) &= K(t)P_{xx}(t) K(t)^{\tp},
    \end{align}
    \end{subequations}
    where $K(t)$ is the associated LQR gain. 
\end{example}

\subsection{Steady State Bounds}
\label{sec:ssBounds}

If the process has converged to a stationary distribution, then the
moments must be constant. This implies that the true state satisfies
$\frac{d}{dt} \XTrue(t) = 0$. 
\begin{theorem}
{\it
    Let $A$, $B$, $H$, $K$, $M$, and $M_{b_i}$ be 
  the terms defined in
  Lemmas~\ref{lem:dynamics}~--~\ref{lem:LMI}. Consider the following 
  semidefinite program:
    \begin{subequations}
      \label{eq:ssSDP}
    \begin{align}
      & \underset{X,U}{\textrm{minimize}} && 
      H X + K U \\
      & \textrm{subject to} && 0 = A X + B U \\
      &&& M(X,U) \succeq 0  \\
      &&& M_{b_i}(X,U) \succeq 0 \quad \textrm{for all } i\in \{1,\ldots,n_b\}.
    \end{align}
  \end{subequations}

  The optimal value for this problem is always a lower bound on the
  optimal value for the steady-state stochastic control problem,
  \eqref{eq:ssControl}, provided that the stationary moments exist and
  are finite. If the number of constraints in the
  linear problem, \eqref{eq:ssSDP}, is increased, either by adding
  more moments to 
  $\XTrue(t)$ or by adding more semidefinite constraints, the value of 
  \eqref{eq:ssSDP} cannot decrease.
}
\end{theorem}

\begin{remark}
  \label{rem:specialCase}
A special case of this theorem is studied in detail in
\cite{KuntzStatMoments16}. Specifically, that work considers the case
in which the stochastic process is given by an uncontrolled stochastic differential
equation with no jumps. They prove that in some special cases that the
upper and lower bounds converge.
Our result is more general, in that it can be
applied to processes with jumps and with control inputs. Furthermore,
we consider both finite horizon problems
and the stationary case. 
\end{remark}

\subsection{Constructing a Feasible Controller}
\label{sec:controller}

The result of the SDPs \eqref{eq:SDP} and \eqref{eq:ssSDP} can be used
to give lower bounds on the true value of optimal control
problems. However, they do not necessarily provide a means for
computing feedback controllers which achieve the bounds. Indeed, aside
from cases with closed moments, the lower bounds cannot be achieved by
feedback. This subsection gives a heuristic method for computing lower
bounds.

The idea behind the controller is to fix a parametric form of the
controller, and then optimize the parameters so that the moments
induced by the controller match the optimal moments as closely as
possible. Assume that $\u(t)$ is a polynomial function of $\x(t)$:
\begin{equation}
  \label{eq:expansion}
\u(t) = \sum_{i=1}^p k_i(t) \x(t)^{(d_i)},
\end{equation}
where $k_i(t)$ is a vector of coefficients and $\x(t)^{(d_i)}$ is a monomial. In this case, the
correlation between $\u(t)$ and any other monomial $\x(t)^{(m)}$ can be
expressed as:

\begin{equation}
  \label{eq:momentCoef}
  \langle \u(t) \x(t)^{(m)} \rangle = \sum_{i=1}^p k_i(t) \langle
  \x(t)^{(d_i+m)} \rangle.
\end{equation}

Our approach for computing values of $k_i(t)$ is motivated by the
following observations for the linear quadratic regulator. 

\begin{example}[LQR -- Continued]
  Recall the regulator problem, as discussed in Examples~\ref{ex:lqr}
  and \ref{ex:lqrSol}. In this case, it is well-known that the true
  optimal solution takes the form $\u(t) = K(t) \x(t)$. This is a
  special case of \eqref{eq:expansion} in which only linear monomials
  are used. In this case, \eqref{eq:momentCoef} can be expressed as:
  \begin{equation}
    P_{ux}(t) = \langle \u(t) \x(t)^\tp \rangle = K(t) \langle \x(t)
    \x(t)^\tp \rangle = K(t) P_{xx}(t). 
  \end{equation}
  Provided that $P_{xx}(t) \succ 0$, the LQR gain can be computed from
  the solution of the SDP \eqref{eq:lqrSDP} by setting $K(t) =
  P_{ux}(t) P_{xx}(t)^{-1}$. 
\end{example}

The regular problem is simple because the moments are closed and the
optimal input is a linear gain. Typically, however, the moments of the
state are not closed and there is no guarantee that the optimal
solution has the polynomial form described in
\eqref{eq:expansion}. For the regulator problem, we saw that the true
gain could be computed by first solving the associated SDP and then
solving a system of linear equations.

We generalize this idea as
follows.
\begin{enumerate}
\item Solve the SDP from \eqref{eq:SDP}.
\item Solve the least squares problem:
  \begin{equation}
    \label{eq:leastSquares}
    \underset{k_i(t)}{\textrm{minimize}}
    \left\|
      \begin{bmatrix}
        \langle \u(t) \x(t)^{(m_1)} \rangle & \cdots
        &
        \langle \u(t) \x(t)^{(m_q)} \rangle 
      \end{bmatrix}
      - \sum_{i=1}^p k_i(t)
      \begin{bmatrix}
        \langle \x(t)^{(d_i+m_1)}\rangle & \cdots &
        \langle \x(t)^{(d_i+m_q)} \rangle
      \end{bmatrix}
    \right\|_{F}^2,
  \end{equation}
  where $\|\cdot\|_F$ denotes the Frobenius norm. 
\end{enumerate}
As long as all of the moments $\langle \u(t) \x(t)^{m_j} \rangle$ and
$\langle \x(t)^{(d_i+m_j)} \rangle$ are contained in either
$\XTrue(t)$ or $\UTrue(t)$, their approximate values will be found in
the SDP \eqref{eq:SDP}. With fixed values from the SDP, the least
squares problem can be solved efficiently for coefficients
$k_i(t)$.

If the SDP gave the true values of the optimal moments, and the true
optimal controller really had polynomial form described in
\eqref{eq:expansion}, then the procedure above would give the exact
optimal controller. When the moments are not exact or the true
controller is not of the form in  \eqref{eq:expansion}, then the
procedure above gives coefficients that minimize the error in the
moment relationships from \eqref{eq:momentCoef}. Intuitively, this
gives coefficients that enforce the optimal correlations between the
controller and state as closely as possible.

\begin{proposition}
  \label{prop:controller}
  {\it
    Let $\u(t)$ be the controller computed according to
    \eqref{eq:expansion} and \eqref{eq:leastSquares}.
    The expected value of the cost \eqref{eq:cost} induced by this
    controller is an upper bound on the optimal value for the original
    optimal control problem
    \eqref{eq:optControl}.
  }
\end{proposition}
\begin{proof}
  The controller produced by \eqref{eq:expansion} and
  \eqref{eq:leastSquares} is feasible. Since \eqref{eq:optControl} is
  a minimization problem, any feasible solution gives an upper bound
  on the optimal value. 
\end{proof}

Combining Theorem~\ref{thm:finiteHorizon} and
Proposition~\ref{prop:controller} gives the following corollary. 

\begin{corollary}
  {\it
  Let $\Lower$ be the optimal value of \eqref{eq:SDP} and let $\Upper$
  be the expected cost induced by the controller computed from
  \eqref{eq:expansion} and \eqref{eq:leastSquares}. If $V^*$ is the
  optimal value of the original optimal control problem,
  \eqref{eq:optControl}, then $V^*$ satisfies:
  \begin{equation}
    \Lower \le V^* \le \Upper.
  \end{equation}
  }
\end{corollary}

\begin{remark}
In Theorem~\ref{thm:finiteHorizon}, we saw that increasing the size of
the SDP leads to a monotonic sequence of lower bounds $\Lower_0 \le
\Lower_1 \cdots \le V^*$. Each of the corresponding SDPs has an
associated upper bound $\Upper_i$, computed via the least squares
method above. We have no formal guarantee that $\Upper_i$ decreases,
but in numerical examples below, we see that it often does. 
\end{remark}

\subsection{Numerical Examples}
\label{sec:num}

This subsection applies the methodology from the paper to the running
examples. The LQR example is omitted, since the results are well
known and our method provably recovers the standard LQR solution.
In the problems approximated via Theorem~\ref{thm:finiteHorizon}, the
continuous-time SDP was discretized using Euler integration. In
principle, higher accuracy could be obtained via a more sophisticated
discretization scheme \cite{raosurvey2009}. 

\begin{example}[Stochastic Logistic Model -- Continued]
  \begin{figure}
    \centering 
    \includegraphics[width=.5\textwidth]{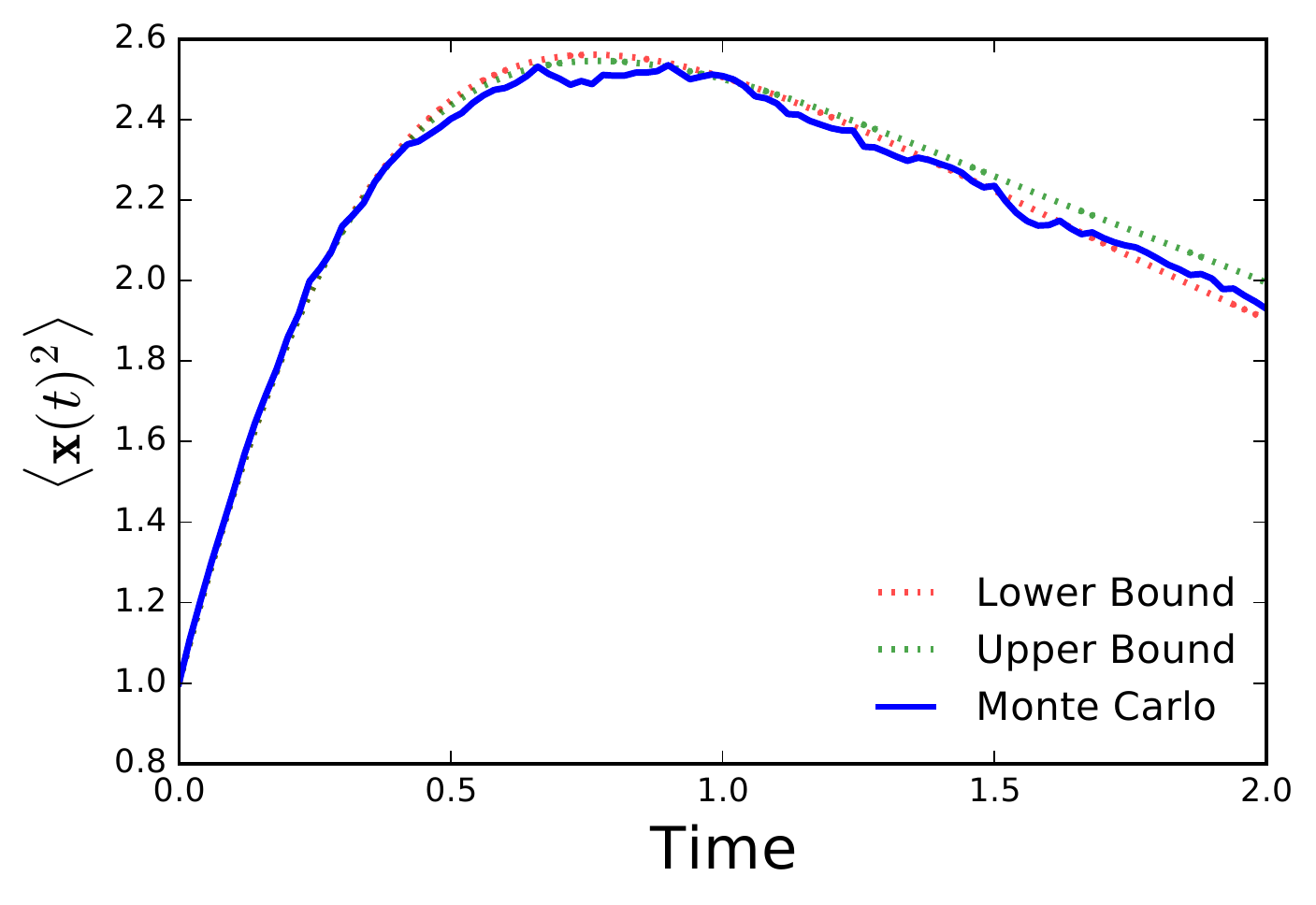}
    \caption{
      \label{fig:logistic}
      The second moment of the logistic model,
      \eqref{eq:stochLog}, using parameters $a_1 = 3$, $b_1=1$,
      $a_2=1$, $b_2=0$ and initial condition $\x(0)=1$. By maximizing and minimizing the bounds on
      $\langle \x(T)^2 \rangle$ using the SDP from
      Theorem~\ref{thm:finiteHorizon}, we get upper and lower bounds
      on the true value. The result of $5000$ runs of the model are
      also shown. While we only penalized the final value of $\langle
      \x(T)^2 \rangle$, the upper and lower bound trajectories have
      similar values over most of the time horizon. Furthermore, these
      trajectories both give good approximations to the average
      trajectory found by Monte Carlo simulations. 
    }
  \end{figure}
    Recall the stochastic logistic model from \eqref{eq:stochLog}. We
  used Theorem~\ref{thm:finiteHorizon} to compute upper and lower bounds on the 
  second moment at the final time, $\langle \x(T)^2 \rangle$. 
  See Fig.~\ref{fig:logistic}. Surprisingly, the bounds give a good
  approximation to $\langle \x(t)^2 \rangle$ for all $t\in [0,T]$,
  despite only optimizing the bound at the final time, $T$. 

  \end{example}

\begin{example}[Randomly Sampled Feedback -- Continued]
  \begin{figure}
    \centering
    \includegraphics[width=.5\textwidth]{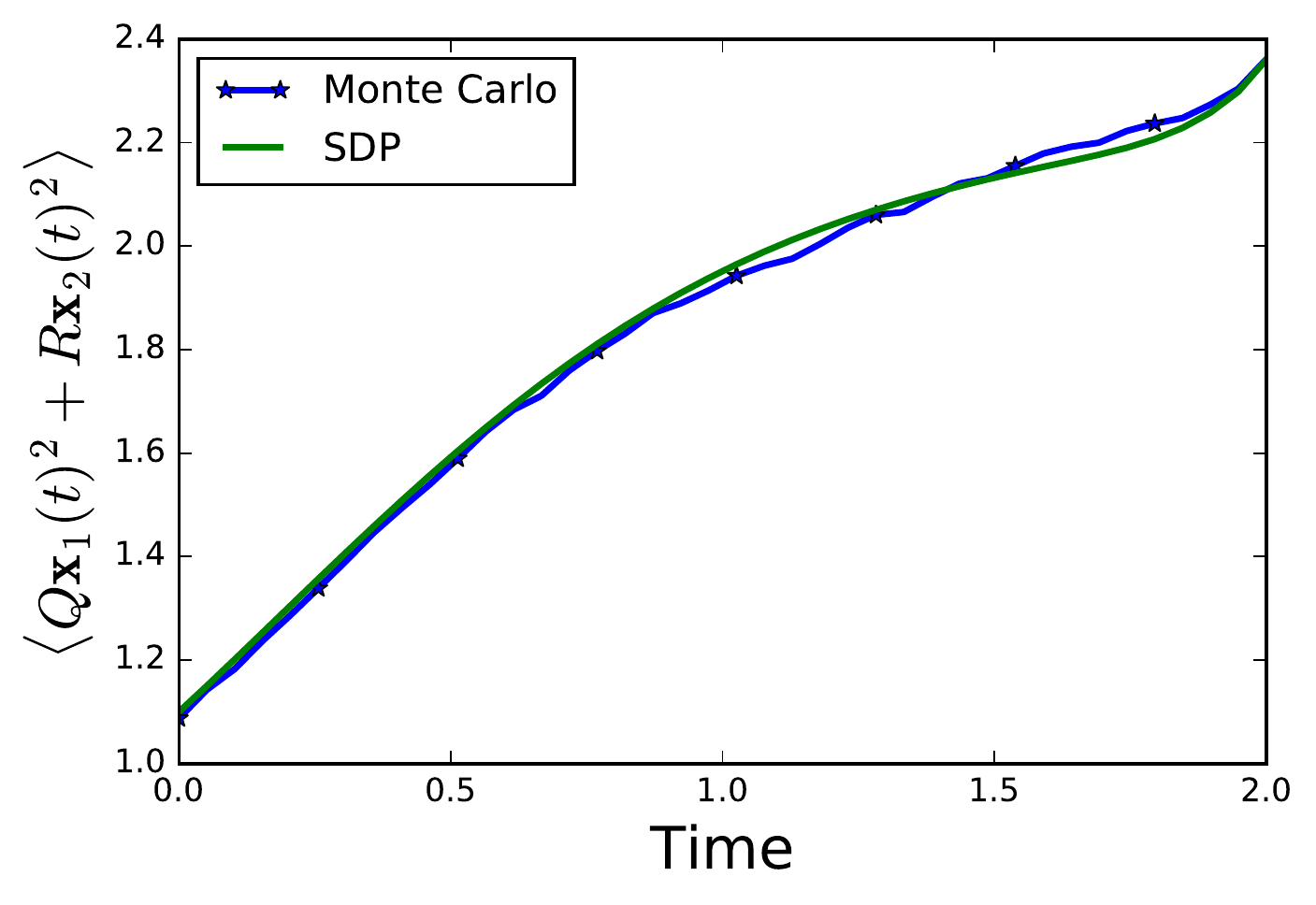}
    \caption{\label{fig:jumpControl} This shows the cost
      $\langle Q\x_1(t)^2 + R\x_2(t)^2\rangle$ averaged over $5000$
      runs of the randomly sampled system defined by
      \eqref{eq:sampledSys} with $Q=1$ and $R = 0.1$. With no feedback, the second moment of $\x_1(t)$ would
      increase linearly from $1$ to $3$, and so we would see the final
    value of the cost to be at least $3$. Note that the SDP cost and
    the Monte Carlo cost are very close. This is to be expected
    because the moments are closed in this example.}
\end{figure}
Recall the sampled-data system from \eqref{eq:sampledSys} and
\eqref{eq:sampledCost}. For $Q=1$ and $R=0.1$, the lower bound from
\eqref{eq:SDP} is $3.67$. A linear controller was constructed using
the method from Subsection~\ref{sec:controller}. This problem has
closed moments at second order, and thus we would expect our method to
find the exact solution. Fig.~\ref{fig:jumpControl} shows the cost
computed from the SDP as well as the cost found by simulating the
computed controller $5000$ times. As expected, the trajectories are
very close. 
\end{example}

\begin{example}[Jump Rate Control -- Continued]
  Recall the jump rate control problem from
  Examples~\ref{ex:rateControl} and
  \ref{ex:jumpAux}. Fig.~\ref{fig:rateControl} shows how the control
  strategy and performance vary as the size of the auxiliary control
  problem increases. As the order of the auxiliary problem increases,
  the computed policy begins to resemble an approximate
  event-triggered policy \eqref{eq:optJumpStrategy} as predicted from
  the HJB equation, \eqref{eq:jumpHJB}. Furthermore, as the degree of
  the controller increases, the computed costs appears to approach the
  steady-state cost predicted from the SDP. 
  
  \begin{figure}
  \centering
  \begin{minipage}[b]{.45\textwidth}
    \centering
    \includegraphics[width=\textwidth]{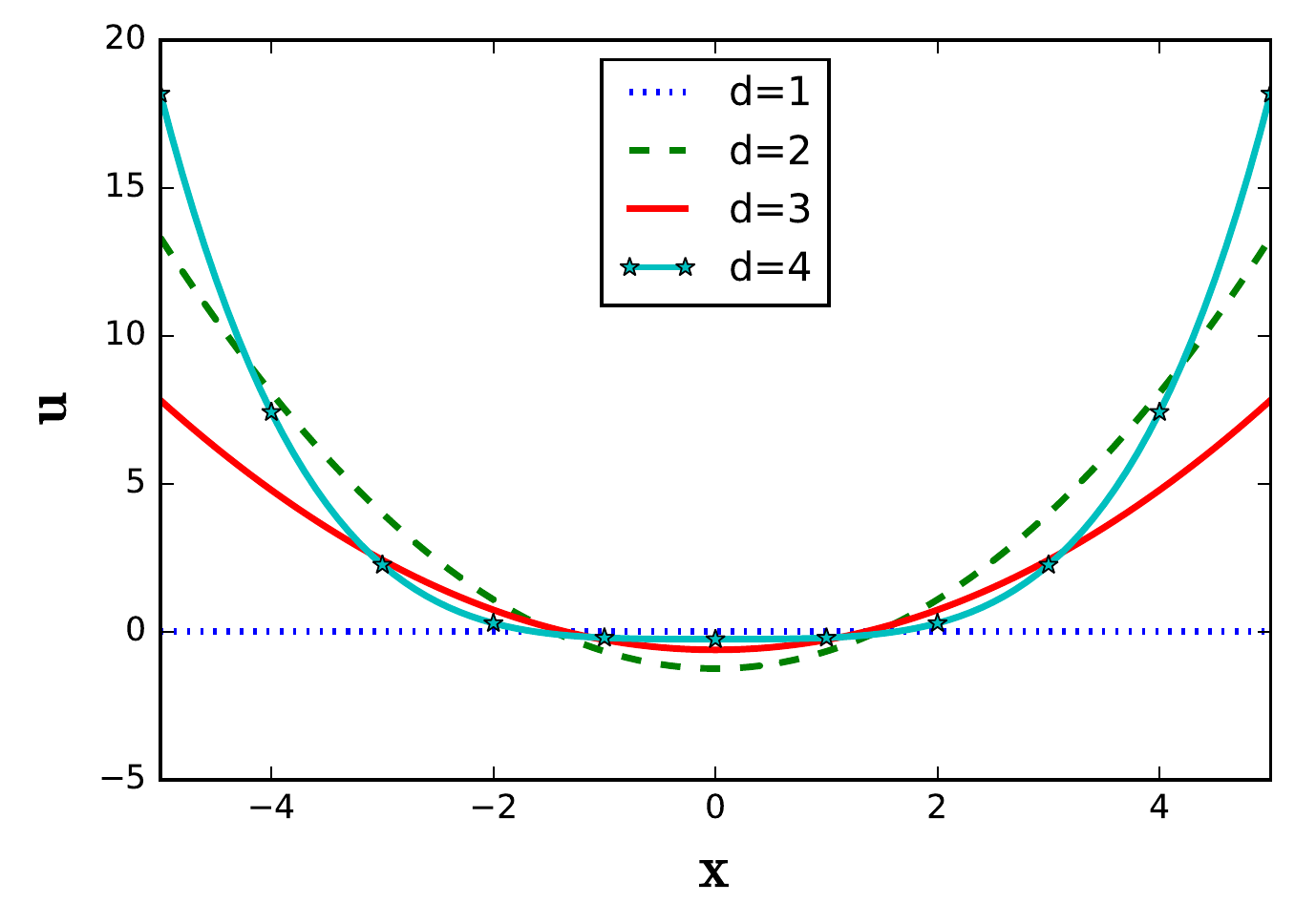}
    \subcaption{\label{fig:rateStrategies} Jump Rate Strategies}
  \end{minipage}
  \begin{minipage}[b]{.45\textwidth}
    \centering
    \includegraphics[width=\textwidth]{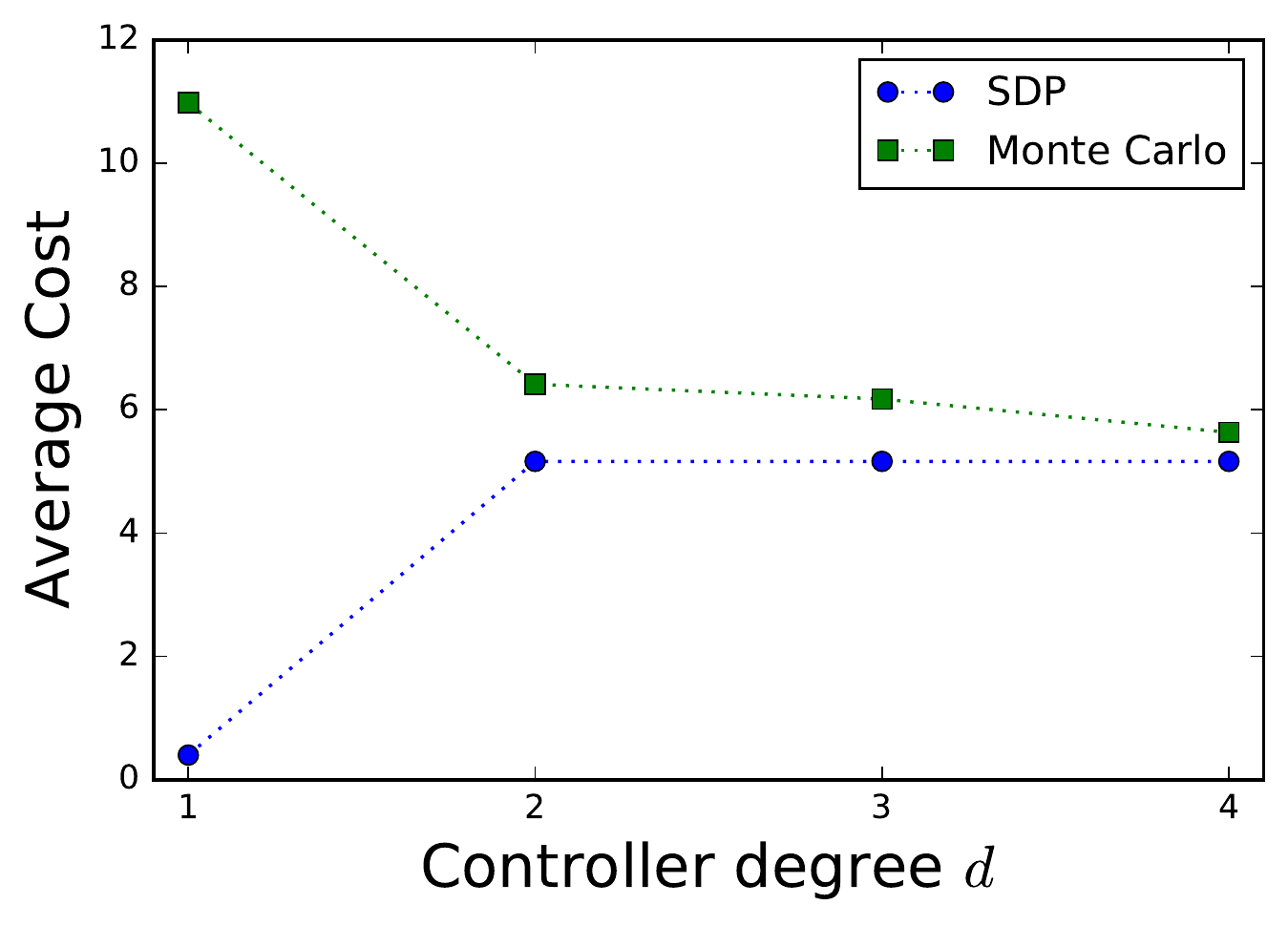}
    \subcaption{\label{fig:rateCosts} Average Cost}
  \end{minipage}
  \caption{\label{fig:rateControl} Fig~\ref{fig:rateStrategies} shows
    the control policies computed as in
    Subsection~\ref{sec:controller} for various 
    orders. Here $Q=1$, $R=10$ and the upper bound on
    the jump rate $\Omega = 10$. 
    Recall from Example~\ref{ex:rateControl} that the true optimal
    controller is an approximate event-trigger control
    strategy. Specifically, the true optimal policy maintains a zero
    jump rate for $\x(t)$ in some region, and then transitions to a
    maximum jump rate of $\u(t) = \Omega=100$ when $\x(t)$ crosses the
    boundary. As the degree of the controller increases, the strategy
    becomes flatter near $\x(t)=0$ and then increases rapidly for
    $|\x(t)| > 5$. As can be seen, the least squares method does not
    enforce the constraint that $0\le \u(t) \le \Omega = 10$. To get a
    feasible controller, the value is simply clipped to stay in the
    correct range. 
    Fig~\ref{fig:rateCosts} shows the lower bound from the
    SDP, and compares it with the upper bound given by the feasible
    controller strategies. As can be seen, the lower bound stops
    increasing by degree $2$. However, the control strategy from
    Fig~\ref{fig:rateStrategies} continues to change. The average
    values from control lead to increasingly tight bounds. The average
    costs were averaged over $5000$ runs of length $5s$, at a sampling
    rate of $100Hz$. 
  } 
\end{figure}
\end{example}

\begin{example}[Fishery Management -- Continued]
  \begin{figure}
    \centering
    \begin{minipage}[b]{.45\textwidth}
      \includegraphics[width=\textwidth]{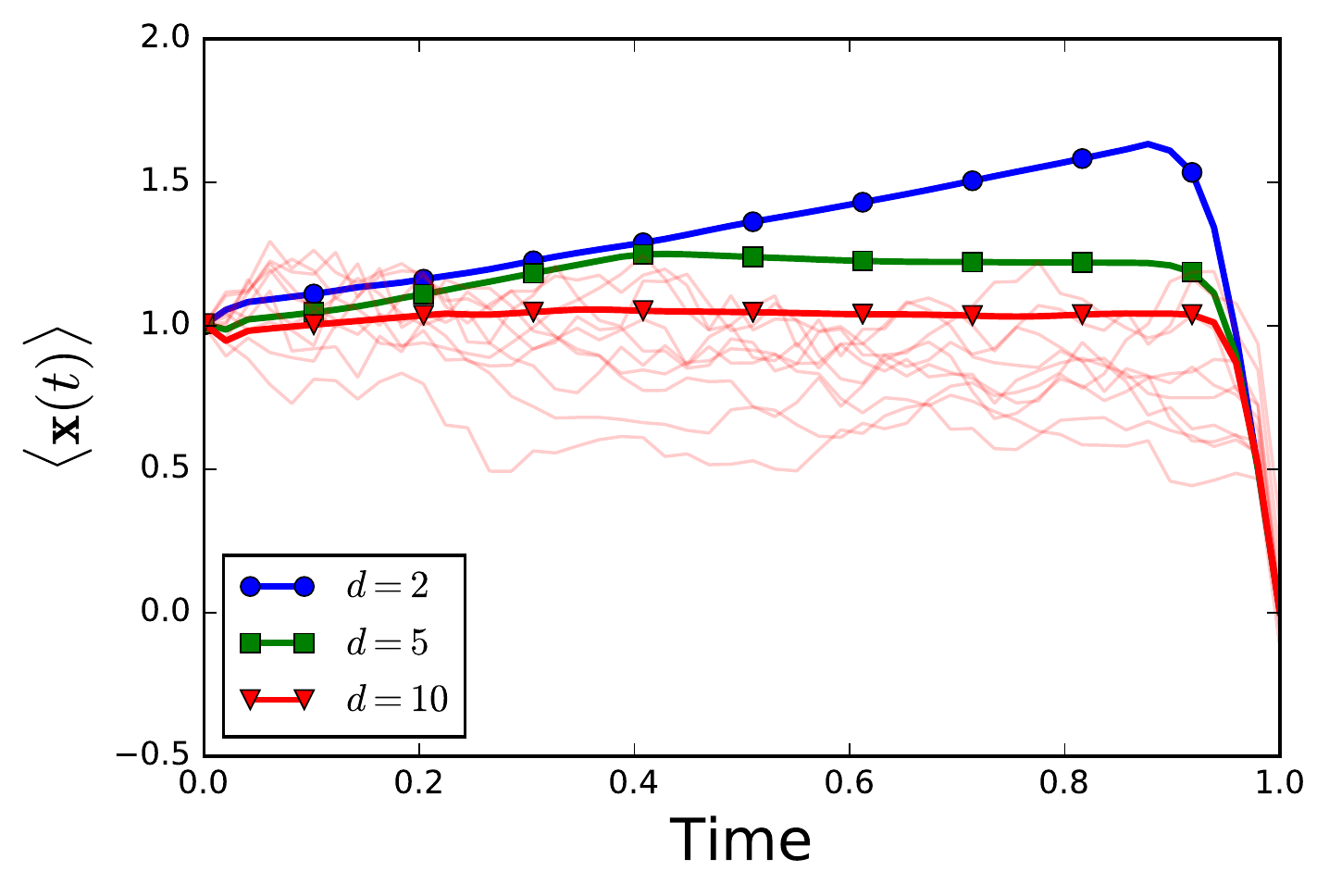}
      \subcaption{\label{fig:harvestPop} Harvesting Population}
    \end{minipage}
    \begin{minipage}[b]{.45\textwidth}
      \includegraphics[width=\textwidth]{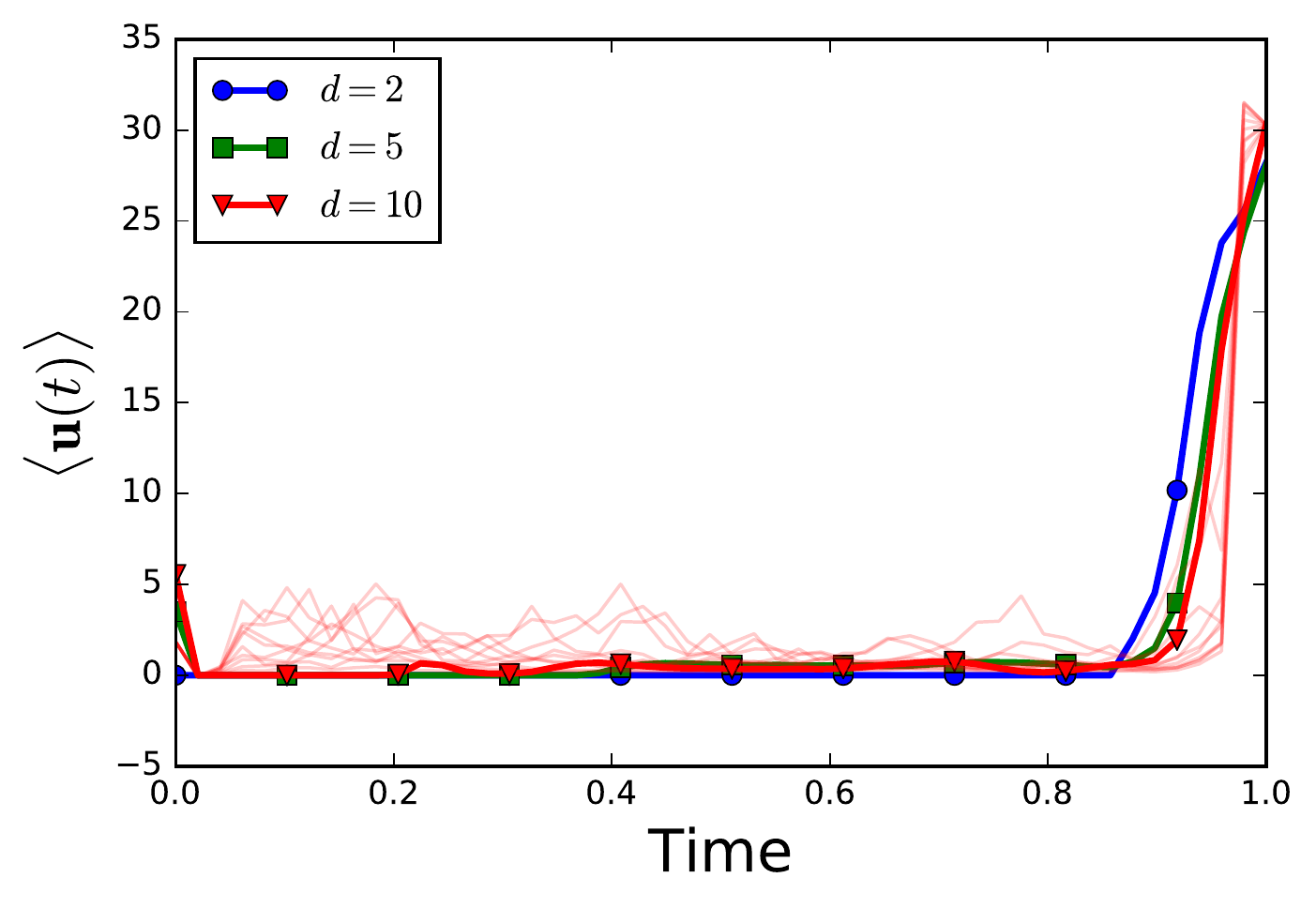}
      \subcaption{\label{fig:harvestStrat} Harvesting Strategy}
    \end{minipage}
    \caption{\label{fig:harvest} Figs~\ref{fig:harvestPop} and
      \ref{fig:harvestStrat} the solutions to the optimal harvesting
      problem. As the order of the controller increases, the following
    strategy emerges. The controller appears to hold the population at
  near $\x(t)=1$ for most of the time. Then, at the end of the
  horizon, the population is harvested to extinction. }
  \end{figure}

  In this example, the smallest upper bound on the total harvest was computed to be $2.12$,
  while the lower bound found be simulating the corresponding
  controller $1000$ times was $1.62$. Thus, the true optimal harvest
  will likely be in the interval $[1.62,2.12]$.
  An exact lower bound
  is not known, since the lower bound was estimated by random
  sampling. In this example, non-negativity of the state and input
  were enforced using the vector inequality constraint defined in
  Lemma~\ref{lem:vec}.

  An interesting control strategy appears to emerge, whereby
  the population is held constant for most of the interval and then
  fished to extinction at the end of the horizon. 
\end{example}

\section{Conclusion}
\label{sec:conclusion}

This paper presented a method based on semidefinite programming for
computing bounds on stochastic process moments and stochastic optimal
control problems in a unified manner. The method is flexible, in that
it can be applied to stochastic differential equations, jump processes
and mixtures of  the two. Furthermore, systems with and without
control inputs can be handled in the same way. The key insight behind
the method is the interpretation of the dynamics of the moments as a
linear control problem. The auxiliary state consists of a collection of moments
of the original state, while the auxiliary input consists of higher order
moments required for the auxiliary state, as well as any terms
involving inputs to the original system. Then all of the desired
bounds can be computed in terms of optimal control problems on this
auxiliary system.

Future work will focus on algorithmic improvements and theoretical
extensions. A simple algorithmic extension would be to use more
general polynomials in the state and input vectors, as opposed to
simple monomials. In particular, this would enable the user to choose
basis polynomials that offer better numerical stability than the
monomials. Methods for automatically constructing the state vectors and
constraint matrices are also desirable. 
In this paper, all of the SDPs were solved using CVXPY \cite{diamondcvxpy2016} in
conjunction with off-the-shelf SDP solvers
\cite{donoghueconic2016,dahl2008cvxopt}. However, the SDP from
\eqref{eq:SDP} has the specialized structure of a linear optimal
control problem with LMI constraints. A specialized solver could
potentially exploit this structure and enable better scaling.
Theoretically, a proof of convergence of the bounds is desirable. For
uncontrolled problems, we conjecture that as long as the moments are all finite,
the upper and lower bounds converge to the true value. Similarly, for
stochastic control problems, we conjecture that the lower bound
converges to the true optimal value.
Furthermore, the general methodology could potentially be extended to
other classes of functions beyond polynomials. In particular, mixtures
of polynomials and trigonometric functions seem to be tractable. More
generally, it is likely that the method will extend to arbitrary basis
function sets that are closed under addition, multiplication, and differentiation.

\bibliography{bibLoc}
\end{document}